\documentclass[11pt]{amsart} \textwidth=14.5cm \oddsidemargin=1cm
\usepackage{amsmath,amsthm,amssymb,latexsym,epic,bbm,comment,yfonts,mathrsfs}
\usepackage{graphicx,enumerate,stmaryrd,rotating}
\usepackage[all]{xy}
\xyoption{2cell}

\allowdisplaybreaks

\usepackage[active]{srcltx}
\usepackage{youngtab}

\usepackage[latin1]{inputenc}
\usepackage{amsxtra}
\usepackage{amsmath}
\usepackage{amscd}
\usepackage{amssymb}
\usepackage{amsfonts}
\usepackage[all]{xy}
\usepackage{mathrsfs}
\usepackage{amsthm}
\usepackage{color}
\usepackage{upgreek}
\usepackage{verbatim}  %%for \begin{comment}
\usepackage{enumerate}
\usepackage{latexsym}
\usepackage{graphicx}
\usepackage{tikz}
\usepackage{todonotes}
\usepackage{setspace}
\usepackage{hyperref}
\usepackage{stmaryrd}
\usepackage{nicefrac}
\usepackage{euscript}
\usetikzlibrary{decorations.pathmorphing}

 \definecolor{darkgreen}{HTML}{336633}
 \definecolor{darkred}{HTML}{993333}
\makeatletter

\newcommand{\arxiv}[1]{\href{http://arxiv.org/abs/#1}{\tt
    arXiv:\nolinkurl{#1}}}

\theoremstyle{plain}
\newtheorem{thm}{Theorem}%[section]
\newtheorem*{thm*}{Theorem}
\newtheorem*{thmA}{Theorem A}
\newtheorem*{thmB}{Theorem B}

\newtheorem{lem}[thm]{Lemma}
\newtheorem{prop}[thm]{Proposition}

\newtheorem{df-prop}[thm]{Definition-Proposition}

\theoremstyle{definition}

\theoremstyle{remark}

\newtheorem{ex}[thm]{Example}
\newtheorem*{exs}{Example}

\def\al{\alpha}

\def\Hom{\operatorname{Hom}\nolimits}

\def\Res{\operatorname{Res}\nolimits}
\def\Ind{\operatorname{Ind}\nolimits}
\def\Ext{\operatorname{Ext}\limits}

\def\gl{\mathfrak{gl}}

\def\la{\lambda}

\def\pn{\mf{pe} (n)}
\def\ov{\overline}

\newcommand{\mc}{\mathcal}
\newcommand{\mf}{\mathfrak}
\newcommand{\C}{\mathbb C}
\newcommand{\oo}{{\ov 0}}
\newcommand{\oa}{{\bar 0}}
\newcommand{\ob}{{\bar 1}}
\newcommand{\vare}{\epsilon} %%%% change-original \vere=\varepsilon

\newcommand{\fg}{\mathfrak{g}}

\newcommand{\fb}{\mathfrak{b}}

\newcommand{\fh}{\mathfrak{h}}

\newcommand{\n}{\mathfrak{n}}
\newcommand{\mZ}{\mathbb{Z}}
\newcommand{\cO}{\mathcal{O}}

\newcommand{\h}{\mathfrak{h}}

\newcommand{\ch}{\mathrm{ch}}
\newcommand{\rad}{\mathrm{rad}}

\newcommand{\Coind}{{\rm Coind}}

\newcommand{\g}{\mathfrak{g}}
\newcommand{\fl}{\mathfrak{l}}
\newcommand{\fp}{\mathfrak{p}}
\newcommand{\fu}{\mathfrak{u}}
\newcommand{\Real}{\mathrm{Re}}

\newcommand{\Z}{{\mathbb Z}}

              \def\wdL{{\widetilde{L}}}
              \def\wdM{{\widetilde{M}}}
            \def\nbob{{\widetilde{M}^\vee}}
            \def\nbobr{{\widetilde{M}_{\mf b^r}^\vee}}
            
\begin{document}

\numberwithin{equation}{section}

\title[On semisimplicity of Jantzen middles]{On semisimplicity of Jantzen middles  for the periplectic Lie superalgebra}

\author{Chih-Whi Chen}
\date{}

\begin{abstract}    We prove that an integral block of the category $\mc O$ of the periplectic Lie superalgebra  contains a non-semisimple Jantzen middle if and only if it contains a simple  module of atypical highest weight. As a consequence, every atypical integral  block of $\mc O$ does not admit a Kazhdan-Lusztig theory in the sense of Cline, Parshall and Scott.  %\cite[Definition 3.3]{CPS93}.
	  
\end{abstract}

\maketitle

%\setcounter{tocdepth}{1}
%\tableofcontents 


\vspace{5mm}

\section{Introduction}\label{sec1}

%  A fundamental problem in the Lie (super)algebra theory is the study of their simple modules.

%{\color{red} category O}

 %Remarkable progress has been made on the study of  the general theory of weight modules; see for example . While the classification of simple weight modules has been extensively studied, the theory of  simple non-weight modules of Lie (super)algebras is still at its beginning stage. 

%We are interested in the following   natural examples: the reductive Lie algebras, the  general linear Lie superalgebras $\gl(m|n)$, the  ortho-symplectic Lie superalgebras $\mf{osp}(2|2n)$ and  the periplectic Lie superalgebras $\pn$.

   \subsection{}

	For a finite-dimensional complex semisimple Lie algebra, it is shown by Andersen and Stroppel \cite[Section 7]{AS} that the validity of Kazhdan-Lusztig conjecture  \cite{KL1} is equivalent to the semisimplicity of Jantzen middles for the regular blocks of the category $\mc O$. Later on, Coulembier  \cite[Theorem 6.4]{Co16} developed   an  analogous connection for basic classical Lie superalgebras. Therefore, the problem of the semisimplicity of Jantzen middles turns out to be interesting and important. 
	
	An earlier achievement is the semisimplicity of Jantzen middles for the general linear Lie superalgebras $\gl(m|n)$ established in \cite[Theorem 6.10]{Co16}, which is  based on the Brundan-Kazhdan-Lusztig theory formulated in \cite{Br1} and proved in \cite{BLW,CLW2}. However, 
 it is still an open question whether the Jantzen middle is always semisimple for Lie superalgebras arising from Kac's classification \cite{Ka1}.
	
	 %For basic classical Lie superalgebras, an implication of the semisimplicity of Jantzen middles in terms of the abstract Kazhdan-Lusztig theory in the sense of \cite{CPS1} has been established  in. 

 	 Recently, the representation  theory for the periplectic Lie superalgebra $\pn$ has been studied extensively; see, e.g., \cite{Se02}, \cite{Ch15}, \cite{B+9}, \cite{Copn}, \cite{CMS}, \cite{EAS1},  \cite{IRS}, \cite{IS}, \cite{KB} and references therein. In addition, basic aspects and partial solutions to the irreducible character problem of the category $\mc O$ were given in \cite{CC} and \cite{CP}. In order to have a complete picture, it is natural to ask whether there exists a Kazhdan-Lusztig pattern for $\pn$.

 	 In \cite{CPS93}, Cline, Parshall and Scott   introduced a formulation of {\em abstract Kazhdan-Lusztig theory}  in order to provide an appropriate axiomatic framework encompassing  numerous important examples in representation theory.  In \cite[Corollary 3.3]{CSTAMS}, it is shown that the  Brundan-Kazhdan-Lusztig theory for the category $\mc O^\Z$ of $\gl(m|n)$-modules of integral weights is an abstract Kazhdan-Lusztig theory.  The goal of this paper is to start the investigation into the semisimplicity of Jantzen middles and connection with Kazhdan-Lusztig theory in the sense of \cite[Definition 3.3]{CPS93} for the periplectic Lie superalgebra $\pn$; see Section \ref{sect::ass} and Section \ref{sect::KLtheory} for the definitions. This type of connection is generalized to the so-called {\em Lie superalgebras of type I} in the full generality.

 \subsection{}	To explain the results of the paper in more detail, we start by explaining our precise
 	setup.   Following \cite{Se11}, we  consider  $\g=\g_\oa\oplus\g_\ob$ a   quasireductive Lie superalgebra (or {\em classical} in \cite{Ma}) throughout the present paper, namely, $\g_\oa$ is reductive and $\g_\ob$ is semisimple over $\g_\oa$ under the adjoint action. In addition, we will make three assumptions in our setup. First, we assume that $\g$ has a type-I gradation $\g =\g_{-1}\oplus \g_0 \oplus \g_{1}$ which is induced by a grading operator from a Cartan subalgebra $\mf h$ of $\g_\oa$.  Next, we choose a triangular decomposition in the sense of \cite{Ma} \begin{align}
 	&\g =\mf n^-\oplus  \h\oplus \mf n^+, \label{eq::12}
 	\end{align} such that the odd parts of $\mf n^\pm$ are $\g_{\pm 1}$, respectively. We refer to these assumptions as \eqref{ass::A1}-\eqref{pre::tridec} in the paper; see  Section \ref{sect::ass}. Lie superalgebras $\g$ satisfying these assumptions fit into the framework of \cite{CC} under the name {\em Lie superalgebras of type I-0}. For such Lie superalgebras, there is a number of basic properties of the twisting functors  developed in \cite[Section 4.3]{CC} that are to be used in the present paper. We will mainly focus on the case of the periplectic Lie superalgebra $\pn$, which is really the main topic  of the paper.
 	 
 	% with a nil-radical $\mf n^+$ coming from a triangular decomposition $\mf g=\mf n^-\oplus \mf h \oplus \mf n^+$ as defined in \cite[Section 2.4]{Ma}.  
 	 
 	  % In the present paper, we will make three assumptions in our setup of Lie superalgebra $\g$. Namely,  we assume that $\g$ is a quasireductive Lie superalgebra with a type-I gradation $\g =\g_{-1}\oplus \g_0 \oplus \g_{1}$ which is induced by a grading operator from a Cartan subalgebra $\mf h$ of $\g_\oa$. In addition, we will fix a triangular decomposition $\g =\mf n^-\oplus  \h\oplus \mf n^+$ such that the odd subalgebras of $\mf n^\pm$ are $\g_{\pm 1}$, respectively; see  Section \ref{sect::ass}.  With slight abusing notation, we again denote by $\mc O$ the corresponding BGG category. We set $\mc O^\Z$ to be the full subcategory of $\mc O$ consisting of modules with integral weights.

Let $\mc O$ denote the BGG category associated to the triangular decomposition \eqref{eq::12}. Let $\mc O^\Z$  be the full subcategory of $\mc O$ consisting of modules with integral weights. Our first main result is the following.
	\begin{thmA}
 Consider $\g$ a Lie superalgebra of type I-0. %(i.e., $\g$ satisfies the assumptions \eqref{ass::A1}-\eqref{pre::tridec}). 
 If a (indecomposable) block of $\mc O^\Z$   contains a (non-zero) non-semisimple  Jantzen middle then it does not admit a Kazhdan-Lusztig theory in the sense of \cite[Definition 3.3]{CPS93}.
\end{thmA}

For $\g$ a (not necessarily type I) basic classical Lie superalgebra from Kac's list \cite{Ka1}, the Theorem A has been established in the earlier work  of Coulembier  \cite{Co16}. % The proof of our Theorem A is an adaption of \cite[Theorem 6.4]{Co16}.

 \vskip 0.1cm

  The notion of {\em typicality} of weights for~$\mf{pe}(n)$ has been introduced in \cite[Section 5]{Se02}; see  \eqref{Sect::421} for its definition.  A block of $\mc O^\Z$ is said to be {\em atypical}, in case it contains a simple module of atypical highest weight, and {\em typical}  otherwise. The following is our second main result. 
	\vskip 0.3cm
		
		\begin{thmB}
			Consider $\g= \pn$. %Let $\mc O_\eta$ be a indecomposable block of the category $\mc O^\Z$. Then $\mc O_\eta$ contains a non-semisimple Jantzen middle if and only if $\mc O_\eta$ contains a simple module of atypical highest weight. 
			Then a block of $\mc O^\Z$   contains a (non-zero) non-semisimple Jantzen middle if and only if  it is atypical. As a consequence, every atypical block of $\mc O^\Z$ does not admit a Kazhdan-Lusztig theory in the sense of \cite[Definition 3.3]{CPS93}.
		\end{thmB}

\subsection{} The paper is organised as follows. In Section \ref{Sect::Pre}, we provide some background materials on quasireductive Lie superalgebras. We review the representation categories and introduce our assumptions  in the Section \ref{sect::ass}. The standard matrix realization of $\pn$ is reviewed in Section \ref{sect::273}.

In Section \ref{sect::Jmiddle}, we study the relevance between semisimplicity of Jantzen middles and  abstract Kazhdan-Lusztig theory in the sense of \cite[Definition 3.3]{CPS1}. %In Section \ref{Sect::Jimddlesub1}, we obtain a sufficient condition  of semisimplicity of Jantzen middles in terms of abstract  Kazhdan-Lusztig theory for any Lie superalgebra. 
Section \ref{Sect::Jimddlesub1} is devoted to the proof of Theorem A. Then, we put these results together  to obtain the proof of Theorem B in Section \ref{sect::KLpn}.

%{\color{blue} The present paper is organized as follows. Realization of twisting functors and Joseph's Enright completion functors. The proof of the main theorem will be given in Subsection \ref{pf::mainthm} }

%To overcome this difficulty 

\subsection*{Acknowledgment}
 The author was supported by a MoST grant, and he would like to thank  Shun-Jen Cheng and  Kevin Coulembier for interesting discussions and helpful comments.

\section{Preliminaries} \label{Sect::Pre}
Throughout the paper the symbols $\C,\mathbb R, \Z, \Z_{\geq 0}$ stand for the sets of all complex numbers, real numbers, integers and non-negative integers. We always work over the ground field $\C$. Denote the abelian group of order two by $\Z_2=\{\oa, \ob\}$. For a homogeneous element $x$ of a vector superspace $V=V_\oa\oplus V_\ob$, we denote its parity by $\overline x\in \Z_2$. In the paper, we let $\g=\g_\oa\oplus \g_\ob$ be a finite-dimensional  quasireductive Lie superalgebra, namely, $\g_\oa$ is reductive and $\g_\ob$ is a semisimple $\mf g_\oa$-module under the adjoint action.  We denote the universal enveloping algebra of $\g$ by  $U(\g)$ and its center by $Z(\g)$. % We will  use the notations $\widetilde U: = U(\mf g)$ and $U:=U(\mf g_\oa)$.

In this section, we collect preliminaries and assumptions  on quasireductive Lie superalgebras.

\subsection{Assumptions and notations} \label{sect::ass}
% Throughout the present paper, we let $\g$ be a  quasireductive Lie superalgebra of {\em type I-0} in the sense of \cite[Section 2.3.1]{CC} with a compatible $\Z$-grading $\g =\g_{-1}\oplus \g_0 \oplus \g_{1}$ induced by a grading element $H\in \h$. Namely, $\g_0 =\g_\oa$ and $\g_\ob =\g_{-1}\oplus\g_1$ with $[\g_1,\g_1] =[\g_{-1},\g_{-1}] =0$. We will use   notations $\g_{\leq 0}:=\g_0\oplus \g_{-1}$ and  $\g_{\geq 0}:=\g_0\oplus \g_{1}$.
 Throughout the present paper, we assume that the Lie superalgebra $\g$ is of  type I-0 in the sense of \cite[Section 2.3.1]{CC}, which we shall explain as follows. %, namely, $\g$ satisfies three assumptions \eqref{ass::A1}, \eqref{ass::A2} and \eqref{pre::tridec}, which we shall explain as follows. 

 \subsubsection{} 
 Fix a triangular decomposition of $\g_\oa$:
 \begin{align}
 &\g_\oa = \mf n_\oa^- \oplus \mf h \oplus \mf n_\oa^+. \label{eq::parad}
 \end{align} In the present paper, we  assume that $\g$  is a  quasireductive Lie superalgebra  with a compatible $\Z$-grading $\g =\g_{-1}\oplus \g_0 \oplus \g_{1}$ induced by a grading element $H\in \h$, that is, 
 \begin{align} 
 &\g_0 =\g_\oa \text{ and } \g_\ob =\g_{-1}\oplus\g_1 \text{ with }[\g_1,\g_1] =[\g_{-1},\g_{-1}] =0. \label{ass::A1}\tag{{\bf A1}}\\
 &[H, x]= kx, \text{ for }x \in \g_k \text{ with }k=\pm 1. \label{ass::A2}\tag{{\bf A2}}
 \end{align} We refer to such a Lie superalgebra as {\em a Lie superalgebra of type I-0}.  We will use   notations $\g_{\leq 0}:=\g_0\oplus \g_{-1}$ and  $\g_{\geq 0}:=\g_0\oplus \g_{1}$.

%For a given triangular decomposition of $\g_\oa$: \begin{align} &\g_\oa = \mf n_\oa^- \oplus \mf h \oplus \mf n_\oa^+ \label{eq::parad}\end{align} with 
For an element $h\in \mf h$,  we define the following subalgebras:
\begin{equation}\label{deflu}\fl:=\bigoplus_{\Real \alpha(h)=0} \fg^\alpha,\quad \fu^+:=\bigoplus_{\Real \alpha(h)>0} \fg^\alpha, \quad \fu^-:=\bigoplus_{\Real \alpha(h)<0} \fg^\alpha,\end{equation}   where $\mf g^\alpha :=\{X\in \g|~[h,X] = \alpha(h)X, \text{~for all }h\in \h\}$. 
We claim that \eqref{ass::A1} and \eqref{ass::A2} imply the following assertion:
\begin{align} 
&\text{There exists an element $h'\in \mf h$ giving rise to $\mf l=\mf h,~\mf u_\oa^\pm =\mf n_\oa^\pm$ and $\mf u_\ob^\pm = \g_{\pm 1}$}. \label{pre::tridec} \tag{{\bf A3}}
\end{align} 
To see this, let $t\in \h$ such that $$\mf h=\bigoplus_{\Real \alpha(t)=0} \fg_\oa^\alpha,\quad \mf n_\oa^+=\bigoplus_{\Real \alpha(t)>0} \fg_\oa^\alpha, \quad \mf n_\oa^-=\bigoplus_{\Real \alpha(t)<0} \fg_\oa^\alpha,$$ then there exists a positive real number $\epsilon$ such that $h':=H+\epsilon t$ gives the desired subalgebras in \eqref{deflu}. Define $\mf n^\pm:=\mf u^\pm$. We refer to the decomposition $\mf g= \mf n^- \oplus \mf h \oplus \mf n^+$ satisfying  \eqref{pre::tridec} as the (distinguished) triangular decomposition of $\g$.

Also, we  refer to the subalgebras
$\mf b:= \mf h+\mf n^+$ and $\mf b^r:= \mf b_\oa +\g_{-1}$ as {\em standard Borel subalgebra} and {\em reverse Borel subalgebra}, respectively. There subalgebras are all  {\em Borel-Penkov-Serganova subalgebras} in the sense of \cite{PS} and \cite[Section 3.2]{Mu12}; see also \cite[Sections 1.3, 1.4]{CCC} for more details. %about triangular decompositions. 

We will make conventional definitions as follows. An element $\alpha\in \h^\ast\backslash\{0\}$ is called a root if $\g^\alpha\neq 0$.  We denote the set of roots by $\Phi\subset \h^\ast$.  Let $\Phi^+$ be the set of roots in $\mf n^+$.  Let $\Phi^+_\oa$ be the positive    system coming from the triangular decomposition \eqref{eq::parad} of $\g_\oa$.  We let $\Pi_0$ be  the corresponding simple system for $\Phi^+_\oa$. We are mainly interested in the following Lie superalgebras:

\begin{exs} \label{ex::ex1} Each of the following quasireductive  Lie superalgebra is of type I-0:
	\begin{enumerate}
		\item[$\bullet$] Reductive Lie algebras $\mf g=\mf g_\oa$. %In this case the good involution $\sigma$ can be chosen as the standard anti-involution which interchanges Chevalley generators and fix elements in $\mf h$; see \cite[Section 0.5]{Hu08}.
		\item[$\bullet$] The general linear Lie superalgebras $\gl(m|n)$; see \cite[Section 1.1.2]{ChWa12}.
		\item[$\bullet$] The ortho-symplectic Lie superalgebras $\mf{osp}(2|2n)$; see \cite[Section 1.1.3]{ChWa12}.
		\item[$\bullet$] The periplectic Lie superalgebras $\mf{pe(n)}$; see Section \ref{sect::273}. %In this case, a explicit good involution has been constructed in \cite[Section 4.3]{CC}.
		\item[$\bullet$]  A semisimple extension
		\begin{align*}
		&\g:=(\mf s\otimes \Lambda (\xi))\rtimes \mf d
		\end{align*} of the Takiff superalgebra induced by a simple Lie algebra $\mf s$  studied in \cite[Section 2.1]{CCo}. 
	\end{enumerate}
\end{exs}

\subsubsection{} The Weyl group $W$ is defined as the Weyl group of $\g_\oa$. We let $w_0\in W$ denote the longest element in $W$. We fix a $W$-invariant bilinear form $\langle ,\rangle$ on $\h^\ast$. Let $\rho$   denote the half-sum of all roots in $\Phi_\oa^+$.  Let $s_\alpha$ be the reflection associated with the   root $\alpha \in  \Phi_\oa^+$. The dot action of $W$ on $\h^\ast$ is defined as $w\cdot \la =w(\la+\rho) -\rho$, for any $\la \in \h^\ast$.

For any $\alpha \in \Pi_0$, we set $\alpha^\vee:= 2\alpha/\langle\alpha, \alpha\rangle$ to be the co-root to $\alpha$; see \cite[Section 0.2]{Hu08}. A weight is called {\em integral} if $\langle \la, \alpha^\vee\rangle \in \Z$, for any $\alpha\in \Phi_\oa^+$. 
We denote by $\mc P \subset\fh^\ast$ the set of integral weights.  A weight $\la$ is said to be  {\em dominant} (resp. {\em anti-dominant}) if $\langle \la+\rho, \alpha^\vee \rangle \notin \Z_{<0}$ (resp. $\langle \la+\rho, \alpha^\vee \rangle \notin \Z_{>0}$), for any $\alpha \in \Phi_\oa^+$. For a given   weight $\la\in \h^\ast$, we let $W_\la$ be the stabilizer subgroup of $\la$ under the dot action of $W$. %, namely, it is the subgroup generated by  $s_\alpha$ with $\langle \la+\rho_\oa, \alpha^\vee\rangle =0$, for $\alpha \in \Pi_0$.

\subsection{BGG category $\mc O$}  \label{sect::pre1}

\subsubsection{} The BGG  category $\mc O$ associated to the triangular decomposition \eqref{pre::tridec} is defined as the category of  finitely-generated $\mf g$-modules on which $\mf h$ acts semisimply and $\mf b$ acts locally finitely. Therefore $\mc O$ is the category of $\g$-modules restricted  to $\mf g_\oa$-modules by $\Res$ in the classical BGG category $\mc O_\oa$ of $\mf g_\oa$-modules as defined in \cite{BGG}. 
Also, we let $\widetilde{\mc F}$ and ${\mc F}$ denote the category of finite dimensional $\g$-modules and $\mf g_\oa$-modules, respectively. 

%The parity shift functor on $\g$-Mod is denoted by $\mc S$. 

%We denote by $\mf g$-Mod  the category of all $\mf g$-supermodules, with parity preserving module morphisms. The parity shift functor on $\g$-Mod is denoted by $\mc S$. Similarly, we define the category $\mf g_\oa$-Mod. Observe that $\mf g_\oa$-Mod is the direct sum of two copies of the usual representation category. 

%We have exact induction and coinduction functors $\text{Ind, Coind}: \mc O_\oa\rightarrow \mc O$ defined by  $$\Ind(-)=U(\mf g)\otimes_{U(\mf g_\oa)}-\qquad\mbox{and}\qquad \Coind(-)=\Hom_{U(\mf \g_\oa)}(U(\mf g),-).$$They are left and right adjoint functors to the restriction $\Res(-): \mc O\rightarrow \mc O_\oa$. By \cite[Theorem~2.2]{BF} (see also \cite{Go}), the functors $\Ind(-)$ and $\Coind(\Lambda^{\text{max}}(\mf g_\ob)\otimes -)$ are isomorphic. Also, 
We have the exact  Kac functor $K(-): \mc O_\oa\rightarrow \mc O$ defined as  $$K(N):=U(\mf g)\otimes_{\mf g_0+\mf g_1} N,$$ for any $N\in \mc O_\oa$ by letting $\g_1$ acts on $N$ trivially.  

For any   $M\in \mc O$, we will freely use   $[M:L]$ to denote the Jordan-H\"older decomposition multiplicities of a simple module $L$ in a composition series of $M$. In addition, we will use $\text{soc}(M)$, $\rad(M)$ to denote the socle and radical of $M$, respectively. The top $M$ is defined as $\text{top}(M):=M/\rad M$.

\subsubsection{} We recall that the category $\mc O$ has a natural structure of highest weight category with respect to the triangular decomposition in \eqref{pre::tridec}. We define the partial order $\le$ on $\fh^\ast$ as the transitive closure of the relations 
\begin{align}
&\lambda\pm\alpha \le\lambda,~\mbox{for $\alpha\in \Phi(\mf n^\mp)$},
\end{align} where $\Phi(\mf n^\mp)$ denotes the set of all roots in $\mf n^\mp$, respectively. For any $\la\in \h^\ast$, we define the Verma module over $\mf g_\oa$ as follows $$M(\la):=U(\mf g_\oa)\otimes_{\fb_{\oa}}\C_\lambda,$$
by  letting $\mf n^+$ acts on  $\C_\la$ trivially.  Also, the corresponding  Verma (super)module over $\g$ is defined as   \begin{align}
&\widetilde{M}(\la) :=  (U(\mf g)\otimes_{\mf b}   \C_\la)\cong K(M(\la)). 
%&\widetilde{M}(\la, \ob) :=\mc S\widetilde{M}(\la, \oa) \cong K(\mc S M(\la)).
\end{align} 
 The (simple) tops of   $M(\la)$ and $\widetilde{M}(\la)$ are denoted by  $L(\la)$ and $\widetilde{L}(\la)$, respectively. Then $\{L(\la)|~\la\in \h^\ast\}$ (resp. $\{\wdL(\la)|~\la\in \h^\ast\}$) forms the complete list of simple modules in $\mc O_\oa$ (resp. $\mc O$).   For any $\la \in\h^\ast$, the Kac induced module  $K(L(\la))$ is an epimorphic image of  $\widetilde{M}(\la)$. By  \cite[Theorem 3.1]{CCC}  $(\mc O, \leq)$ is a highest weight category with standard objects $\wdM(\la)$. Also, we denote by $P(\la)$ and $\widetilde{P}(\la)$ the projective covers of $L(\la)$ and $\widetilde{L}(\la)$ in $\mc O_\oa$ and ${\mc O},$ respectively. Finally, for $M\in \mc O$, we use $\ch M$ to denote the formal character of $M$.

%\begin{lem} The category $\mc O$ is an abelian category. \end{lem} \begin{proof} 	Matatis mutandis \cite[Lemma 2.2]{Br}. \end{proof} 

\subsubsection{}  Define an involution  $\widehat{(\cdot)}:\h^\ast \rightarrow \h^\ast$  by letting  $\widehat \la =-w_0\la$, for $\la \in\h^\ast$.  As   observed in  \cite[Section 1.3]{CCC}, there is an  anti-involution $\sigma$ on $\g$ satisfying that 
\begin{align}  
&\sigma(\mf h)= \mf h,~\sigma(\mf n^+_\oa) =\mf n_\oa^-, \label{eq::gdinv} \\
&(\widehat \la)(h) = \la(\sigma(h)),\text{ for } h\in \h. 
\end{align} 
 
This involution $\sigma$ leads to a natural duality functor $D$ on the category $\mc O$ as follows. For any $M\in \mc O$,    let    $M^\oast$  be the 
restricted dual space of $M$. Then $M^\oast$ is a $\g$-submodule of $\text{Hom}_\C(M,\C)$. We now give a new  $\mf g$-module structure of $M^\oast$ by letting 
\begin{align}
&xf(v) = (-1)^{\overline x \overline f} f(\sigma(x)v),
\end{align}  for any homogeneous elements $x\in \g$, $f\in M^\oast$ and any $v\in M$.  Then denote this resulting  module by $DM$. This gives the endofunctor $D$ on $\mc O$; see \cite[Section 2.2.4]{CC} for more details.

By \cite[Section 3]{CCC}, the functor $D$ intertwines the standard and costandard objects of $\mc O$ with respect to the two highest weight category structures via $\mf b$ and $\mf b^r$. We briefly recall this effect below. 

 For any $\mu\in \h^\ast$,  let $\nbob(\mu)$ be the maximal submodule of the $\Coind_{\mf n_\oa^-+\g_{-1}}^{\mf g}(\C_{\mu})$  on which $\mf h$ acts semisimply and locally finitely; see \cite[Definition 3.2, Theorem 3.1]{CCC}. Next, we put  $\wdM_{\mf b^r}(\mu):= U(\g)\otimes_{\mf b^r}\C_\mu$, and define  $\nbobr(\mu)$ as the maximal submodule of the coinduced module $\Coind_{\mf n_\oa^-+\g_1}^{\mf g}(\C_{\mu})$ on which $\mf h$ acts semisimply and locally finitely. Then by \cite[Proposition 3.4]{CCC}, we have the following isomorphisms
\begin{align}
&D\wdM(\mu)\cong \nbobr(\widehat \mu),~D\nbob(\mu)\cong \wdM_{\mf b^r}(\widehat \mu),\text{ and }D\wdL(\mu)\cong \widetilde L_{\mf b^r}(\widehat \mu), \label{eq::dualVerma}
\end{align}  where $\widetilde L_{\mf b^r}(\widehat \mu)$ denotes the simple highest weight module of highest weight $\widehat \mu$ with respect to  $\mf b^r$.

%\subsection{Freeness of action of root vectors on modules}
For any $\alpha \in \Pi_0$, we consider $f_\alpha \in \mf g^{\alpha}_\oa$ to be a non-zero root vector of root $\alpha$. A $\g$-module $M$ is said to be {\em $\alpha$-finite} (resp. {\em $\alpha$-free}) if the action of  $f_\alpha$ on $M$ is locally finite (resp. injective).  %Similarly, for a given set $\ups \subset \Pi_0$,  the module  $M$ is called $\ups$-finite (resp. $\ups$-free) if $M$ is $\alpha$-finite (resp. $\ups$-free), for any $\alpha \in \ups$.
The following useful lemma is a consequence of \cite[Lemma 2.1]{CoM1}
\begin{lem}  \label{lem::3}
	Let $\la\in \h^\ast$ and $\alpha \in \Pi_0$. Then we have  
	\begin{align}
	&\text{$\wdL(\la)$ is $\alpha$-finite} \Leftrightarrow\text{$D\wdL(\la)$ is $\widehat\alpha$-finite}\Leftrightarrow \text{	$\langle \la+\rho, \alpha^\vee  \rangle \in \Z_{> 0}$}.
	\end{align}
\end{lem}

\subsection{The periplectic Lie superalgebra $\pn$}  \label{sect::273}
In this subsection, we introduce the periplectic Lie superalgebras $\pn$; see also \cite[Section 1.1]{ChWa12}  for more details.

\subsubsection{Matrix realization}
For any positive integer $n$, the standard  matrix realization of the periplectic Lie superalgebra 
$\pn$ inside the general linear Lie superalgebra  
$\mathfrak{gl}(n|n)$ is given by
\begin{align}\label{plrealization}
\pn:=
\left\{ \left( \begin{array}{cc} A & B\\
C & -A^t\\
\end{array} \right)\| ~ A,B,C\in \C^{n\times n},~\text{$B$ symmetric and $C$ skew-symmetric} \right\}.
\end{align}  The type-I gradation of $\pn$ inherits that of $\gl(n|n)$, namely,  
\begin{align*}
\pn_1:=
\{\begin{pmatrix}
0 & B \\
0 & 0
\end{pmatrix}|B^t=B\}\quad\mbox{and}\quad \pn_{-1}:=
\{\begin{pmatrix}
0 & 0 \\
C & 0
\end{pmatrix}|C^t=-C\}.
\end{align*}

The standard Cartan subalgebra $\mf h    \subset \mf g_\oo$ consists of diagonal matrices. 
Let $E_{ab}$ denote the elementary matrix in $\mathfrak{gl}(n|n)$, for $1\leq a,b \leq 2n$. We denote by   $\{\vare_1, \vare_2, \ldots, \vare_n\}$ the dual basis of $\mf h^*$  with respect to the following  standard basis of $\mf h$ 
\begin{align}
\{H_i:=E_{i,i}-E_{n+i,n+i}|~1\leq i \leq n \}\subset \pn. \label{eq::cartan}
\end{align}  
In particular, we have 
\begin{align}\label{eqroots}
&\Phi=\{\epsilon_i-\epsilon_j,~\pm(\vare_i+\vare_j)\,|\, 1\le i\neq j\le n\}\cup \{2\vare_i|~1\leq i\leq n\}, \\
&\Phi_\oa^+=\{\epsilon_i-\epsilon_j\,|\, 1\le i<j\le n\}, \\
&\Pi_0=\{\epsilon_i-\epsilon_{i+1}\,|\, 1\le i\le n-1\}.
\end{align} 
The Weyl group $W$ is isomorphic to the symmetric group on $n$ letters.  We fix a non-degenerate $W$-invariant bilinear form $\langle\cdot, \cdot\rangle: \mf h^*\times \mf h^* \rightarrow \C$ by letting $\langle\vare_i, \vare_j\rangle =\delta_{ij}$, for all $1\leq i,j\leq n$.  Fix the Borel subalgebra $\fb_{\oa}$ of $\g_{\oa}\cong \mathfrak{gl}(n)$
 consisting of matrices in \eqref{plrealization} with $B=C=0$ and
 $A$ upper triangular. Without loss of generality, we shift the Weyl vector $\rho$ of $\g_\oa$ by letting  
$$\rho:= (n-1)\vare_1 +(n-2)\vare_2+\cdots +\vare_{n-1}.$$ Also, we define $\omega_n:=\vare_1+\vare_2+\cdots+\vare_n$. For any $k\in \C$, we denote by $\C_{k\omega_n}$ the one-dimension $\g$-module of weight $k\omega_n$. Note that  $\C_{k\omega_n} \otimes -$ leads to an auto-equivalence of $\mc O$; see \cite[Section 5.10]{CC}.

\subsubsection{Odd reflections}
In this subsection, we recall the notion of odd reflections for $\pn$ from \cite[Lemma 1 and Section 2.2]{PS89}. For a given Borel subalgebra $\mf b'$, we denote the set of  roots of $\mf b'$ (i.e. non-zero weights of $\mf b'$) by $\Phi(\mf b')$. Consider the following sequence
\begin{align}
&\{\alpha_0,\al_1,\ldots,\al_{k-1}\} = \{2\vare_1,\vare_1+\vare_2,2\vare_2, \vare_1+\vare_3,\vare_2+\vare_3,2\vare_3,\ldots,\vare_{n-1}+\vare_n, 2\vare_n\}, \label{eq::oddrefeq0}
\end{align}where $k= \frac{n(n+1)}{2}.$ Let $\mf b^r=\mf b^0, \mf b^1, \mf b^2,\ldots,\mf b^k=\mf b$ be the corresponding sequence of Borel subalgebras in the sense of \cite[Section 2.2]{PS89}. More precisely, $\mf b_\oa^\ell = \mf b_\oa$ for each $0\leq \ell\leq k$ and 
	\begin{align}
&\Phi(\mf b^{\ell+1})=\begin{cases} \Phi(\mf b^\ell)\backslash\{-\alpha_\ell\}\cup\{\alpha_\ell\} &\mbox{ if $\alpha_\ell$ is of the form $\vare_p+\vare_q$ with $p\neq q$,}\\
 \Phi(\mf b^\ell)\cup \{\alpha_\ell\}&\mbox{ otherwise.} 
\end{cases} \label{eq::oddrefeq}
\end{align}
Following \cite{PS89}, two Borel subalgebras $\mf b^{\ell+1}$ and $\mf b^\ell$ are said to be {\em connected by an odd reflection associated to $\alpha_\ell$}, in case  they satisfy the first assumption in \eqref{eq::oddrefeq}, and {\em connected by an inclusion} otherwise. Using the sequence \eqref{eq::oddrefeq0}, we give the rule of change of highest weights of an irreducible module under  the odd reflection and inclusion in following lemma, which is taken from \cite[Lemma 1 and Section 2.2]{PS89}.
\begin{lem}\label{lem::oddref}
Let $\la =\sum_{i=1}^\n \la_i \vare_i\in \mf h^\ast$. Suppose that $$\la^0,\la^1,\la^2,\ldots, \la^k\in \h^\ast$$ are $\mf b^0,\mf b^1,\mf b^2,\ldots, \mf b^k$-highest weights of $\wdL(\la)$. For each $0\leq \ell\leq k$, we set $\la^\ell = \sum_{i=1}^n \la_i^\ell\vare_i$.  Then we have the following rules:
\begin{itemize}
	\item[(A)] For each $\alpha_\ell= \vare_p+\vare_q$ with $p\neq q$, the $\mf b^{\ell+1}$-highest weight of $\wdL(\la)$ is given by 
		\begin{align}
	&\la^{\ell+1}=\begin{cases} \la^\ell+\alpha_\ell &\mbox{ if $\la^\ell_p\neq \la^\ell_q$,}\\
	\la^\ell&\mbox{ otherwise.} 
	\end{cases} 
	\end{align}
	\item[(B)] If $\alpha_\ell =2\vare_p$, then the $\mf b^{\ell+1}$-highest weight of $\wdL(\la)$ is given by $\la^{\ell+1}=\la^\ell$.
\end{itemize}
	\end{lem}

For a given  $\la\in \h^\ast$ and $0\leq \ell\leq k$, we denote by $\wdL_{\mf b^\ell}(\la)$ the simple module of $\mf b^\ell$-highest weight $\la$. 

\begin{ex}\label{ex::example3}
 Consider a weight $\la=\sum_{i=1}^n\la_i\vare_i\in \bigoplus_{i=1}^n \Z\vare_i$ such that $\la_1=\la_2$ and  $\la_2<\la_3<\la_4<\cdots <\la_n$. By a direct computation, it follows that $$\wdL_{\mf b^r}(\la) = \wdL(\la+(n-1)\omega_n -\vare_1 -\vare_2),$$
 namely, $\la^{\ell+1} = \la^\ell+\alpha_\ell$ at each step given in Part (A) of  Lemma \ref{lem::oddref} for Borel subalgebras $\mf b^\ell$ and $\mf b^{\ell+1}$ that are connected by an odd reflection $\alpha_\ell \neq \vare_1+\vare_2$.  
\end{ex}

%\subsubsection{Typicality of weights} For $\g:=\gl(m|n), \mf{osp}(2|2n)$, we recall the notion of atypicality from \cite[Section 2.2.6]{ChWa12}. A (odd) root $\alpha$ is called {\em isotropic} if $\langle \alpha ,\alpha\rangle=0$. Then $\la\in \h^\ast$ is said to be {\em atypical}, in case $\langle \la+\rho ,\alpha\rangle=0$, for some odd isotropic root $\alpha$, and {\em typical} otherwise.  	For $\g = \pn$, we will follow  the notion of typicality for~$\mf{pe}(n)$ defined in~\cite[Section 5]{Se02}; see Section 4.2.1. 

  \section{The Jantzen middles} \label{sect::Jmiddle}
  We continue to assume that $\mf g$ is a quasireductive Lie superalgebra of type I-0 (i.e., $\g$ satisfies  assumptions \eqref{ass::A1}-\eqref{pre::tridec}). Recall that $\mc O^{\Z}$ denotes the full subcategory of $\mc O$ consisting of modules of integral weights. Similarly, we define $\mc O^\Z_\oa\subset \mc O_\oa$.  We will follow   \cite[Section 6]{Co16} and define the Jantzen middles as the radicals of twisted simple modules. Before giving the precise definitions, we recall the twisting functors as follows.

 % \subsection{Realization for completion and twisting functors} \label{Sect::Realization}   In \cite{KM2}, Khomenko and Mazorchuk developed the  realizations of Joseph's completion functors and  Arkhipov's twisting functors for Lie algebras using the partial approximation functors and coapproximation functors. Such realizations were also obtained in \cite[Proposition 5.9]{Co16} for basic classical Lie superalgebras. In this section, we will generalize  these results to Lie superalgebras concerned in the present paper. We assume that $\g$ is a Lie superalgebra satisfying assumptions \ref{ass::A1} - \ref{pre::tridec}.  

  \subsection{Twisting functors}

Let $\alpha\in \Pi_0$ and $s:=s_\alpha$, we recall the corresponding  Arkhipov's twisting functor  $T_s$ (resp. $T_s^\oa$) on $\mc O$ (resp. $\mc O_\oa$)  introduced in \cite[Section 3.6]{CMW} and  \cite[Section 5]{CoM1}. This functor was originally defined by Arkhipov in \cite{Ark1} and further studied in \cite{AS, AL, CMW, KM2, MaSt}. 

Let $G_s$ be the right adjoint to $T_s$. Let us recall some basic properties of $T_s$; see also \cite{AS, CMW}, \cite[Section 5]{CoM1} and \cite[Theorem 4.5]{CC}.
\begin{itemize}
	\item[(1)] The functor $T_s$ is right exact. 
	\item[(2)] The functor $G_s$ is left exact and  isomorphic to the  Joseph's version of Enright completion functor as introduced in \cite[Section 4.2]{CC}. 
	\item[(3)] Let $G_s^\oa$ denote the Joseph's     Enright completion functor  on $\mc O_\oa$   introduced in \cite[Section 2]{Jo82}. Then  we have  \begin{align*}
	&\Ind \circ T_s^\oa=T_s\circ \Ind \text{ and }~\Res \circ T_s=T^\oa_s\circ \Res.\\
	&\Ind \circ G_s^\oa=G_s\circ \Ind \text{ and }~\Res \circ G_s=G^\oa_s\circ \Res.
		\end{align*}  
	\item[(4)] We have $D\circ G_{s_{\widehat \alpha}}\circ D \cong T_{s_\alpha}$ on $\mc O^\Z$.
	\item[(5)] Denote the left derived functor of $T_s$  by $\mc LT_s$. Then $\mc L_iT_s=0,$ for $i>1$. For any $M\in \mc O$, the $\mc L_1T_s(M)$ is the maximal $\alpha$-finite submodule of $M$. If $\wdL(\la)$ is   $\alpha$-finite,  then we have  $T_s\wdL(\la) =0$ and  $\mc L_1T_s\wdL(\la) = \wdL(\la)$. 
	\item[(6)]    Denote the right derived functor of $G_s$  by $\mc RG_s$. Then $\mc LT_s$ is an auto-equivalence of the bounded derived category $\mc D^b(\mc O)$ with $\mc RG_s$ as its inverse. 
\end{itemize}

  % We recall that   the completion functor $G_s(-): \mc O^\Z\rightarrow \mc O^\Z$ from \cite[Section 4.2]{CC}, which is an analogue of Joseph's version of Enright completion $G_s^\oa$ on $\mc O_\oa$ as introduced in \cite[Section 2]{Jo82}.  The functor  $G_s(-)$ is defined as the following composition:  \begin{align}   &G_s(-):=\mc L(M(s\cdot 0), -)\otimes_{U(\g_\oa)} M(0): \mc O^\Z\ \rightarrow \mc B_0 \xrightarrow{\sim} \mc O^\Z.   \end{align} 

  %Following Joseph's paper \cite{Jo82}, the natural embedding $M(s\cdot 0) \hookrightarrow M(0)$ gives rise to a natural transformation $\mc L(M(0), -)\hookrightarrow \mc L(M(s\cdot 0), -)$. Consequently, we obtain a natural transformation $\eta^\alpha: \Id_{\mc O^{\Z} } \rightarrow G_{s}$ by Lemma \ref{CorEqiv2}. By definition, we also have  $\Res \circ G_s=G^\oa_s\circ \Res$, where $G_s^\oa$ is the the classical Joseph's Enright completion functor $G^\oa_s$ from \cite{Jo82}. We may observe that the $\Res\circ \eta^s$ is the classical natural transformation between the identity functor on $\mc O^\Z_\oa$ to $G^\oa_s$ on $\mc O^\Z_\oa$  as introduced in \cite{Jo82}.

  %For a category $\mc C$, we denote by $\Id_{\mc C}$ the identity functor on $\mc C.$ By \cite[Lemma 2.4]{Jo82}, there is a natural transformation $\eta^\alpha_\oa$  from $\Id_{\mc O^{\Z}_\oa}$ to $G_{s}^\oa$ with kernel given by the functor of taking the largest $\alpha$-finite submodule. The following lemma shows that $\eta_\oa^\alpha$ can be lifted to an natural transformation between $\Id_{\mc O^{\Z} }$ and $G_s$.

  Since   twisting functors satisfy the braid relations, see, for example \cite{KM2,CoM1}, it follows that for any $w\in W$ with a reduced expression $w=s_{\alpha_1}s_{\alpha_2}\cdots s_{\alpha_k}$  ($\alpha_1,\ldots , \alpha_k\in \Pi_0$) the associated twisting functor  $T_{w}:=T_{s_{\alpha_1}}\circ T_{s_{\alpha_2}}\circ \cdots \circ T_{s_{\alpha_k}}$ is well-defined.  We use $T^\oa_w$  to denote the corresponding twisting functor  on $\mc O_\oa$. Then we have  $\Res \circ T_w=T^\oa_w\circ \Res$. The completion functors $G_w$ and $G_w^\oa$ are defined in similar fashion.

  \subsection{The Jantzen middles for type-I Lie superalgebras}	\label{Sect::Jimddlesub1}
  Let $\alpha \in \Pi_0$.  Following  \cite[Section 6]{Co16}, we define the {\em Jantzen middle $U_\alpha(\la)$ for $\wdL(\la)$ associated with $\alpha$} as the radical of $T_{s_\alpha}\wdL(\la)$, for any $\la \in \mf h^\ast$.  The following   realization of Jantzen middles is an analogue of \cite[Proposition 6.2]{Co16}, where the case of basic classical  Lie superalgebras were considered. 
  
  \begin{prop} \label{prop::prop10} Consider $\g$  a  quasireductive Lie superalgebra  of type I-0.  
  	Suppose that $\la \in \mc P$ such that $\wdL(\la)$ is $\alpha$-free. Then $T_{s_\alpha} \wdL(\la)$ has a simple top isomorphic $\wdL(\la)$. Furthermore, the Jantzen middle $U_\alpha(\la)$ is isomorphic to the largest $\alpha$-finite quotient of $\rad{\widetilde P(\la)}$. 
  \end{prop}

   \subsubsection{Semisimplicity of Jantzen middles}  For  $\g$ a Lie superalgebras with the category $\mc O$ that admits a simple-preserving duality,  the structure of $T_{s_\alpha} \wdL(\la)$ has been studied in \cite[Theorem 5.12, Corollary 5.14]{CoM1}.  It is shown that $\text{soc} (U_\alpha(\la))\cong \text{top} (U_\alpha(\la))$ in the case when $\g$ is either reductive or basic classical; see   \cite[Theorem 6.3]{AS} and \cite[Corollary 5.14]{CoM1}.  For $\g = \mf{pe}(2)$, we give an example below showing that the socle and radical of a Jantzen middle are not necessarily isomorphic. Instead, we have the following proposition, which is an analogue of \cite[Part (3) of Thoerem 6.3]{AS} and \cite[Part (ii) of Theorem 5.12]{CoM2} for any Lie superalgebra $\g$ of type I-0. In particular, this applies to the case of   $\pn$. 
  
  \begin{prop} \label{prop::100} 	Suppose that $\alpha\in \Pi_0$ and $\la \in \mc P$ such that $\wdL(\la)$ is $\alpha$-free. Then  we have 
  	\begin{align}
  	&\emph{soc} (U_\alpha(\la)) \cong \bigoplus_{\widetilde L(\nu):\emph{~$\alpha$-finite}} \widetilde L(\nu)^{\oplus \emph{dim}\Ext^1_{\mc O}(\widetilde L(\nu), \widetilde L(\la))}.\\
  	&\emph{top} (U_\alpha(\la))  \cong \bigoplus_{\widetilde L(\nu):\emph{~$\alpha$-finite}} \widetilde L(\nu)^{\oplus \emph{dim}\Ext^1_{\mc O}(\widetilde L(\la),\widetilde L(\nu))}.
  	\end{align}
  \end{prop}
    \begin{proof}
    Mutatis mutandis the proof of \cite[Theorem 6.3]{AS}. 
    \end{proof}

  \begin{ex} \label{ex::exm11}
  	Consider $\g=\mf{pe}(2)$ and $\la=2\vare_2\in \h^\ast$. Set $\alpha:=\vare_1-\vare_2$.  We are going to show that the socle and radical of $U_\alpha(\la)$ are not isomorphic.  Recall that $\omega_2:= \vare_1+\vare_2$. 	By \cite[Lemma 5.11]{CC} it follows that  $\wdM(\la) =K(L(\la)) = \wdL(\la)$. We provide two methods to prove the conclusion.
  	
  	{\em Method 1}. % Therefore, we have  $\ch T_{s_\alpha} \wdL(\la) = \wdM(s_\alpha\cdot \la)$ by 
  	We have $\ch T_{s_\alpha} \wdL(\la) = \ch \wdM(s_\alpha\cdot \la)= \ch \wdL(\la)+\ch \wdL(s_\alpha \cdot \la)+\wdL(s_\alpha\cdot \la -\omega_2)$ by \cite[Lemma 6.5]{CC}. It follows that 
  	\begin{align}
  	&\ch U_\alpha(\la) =\ch \wdL(s_\alpha \cdot \la) +\ch \wdL(s_\alpha\cdot \la -\omega_2)=\ch \wdL(\omega_2)+\ch \wdL(0).
  	\end{align}
  	We claim that the socle of $U_\alpha(\la)$ is $\wdL(0)$. To see this, we consider the following short exact sequence
  	\begin{align}
  	&0\rightarrow \wdL(\la) \rightarrow \wdM(\omega_2) \rightarrow K(L(\omega_2)) \rightarrow 0,  
  	\end{align} obtained by applying the Kac functor $K(-)$ to the short exact sequence $0\rightarrow L(\la)\rightarrow M(\omega_2) \rightarrow L(\omega_2)\rightarrow 0$. Since the socle of $K(L(\omega_2))$ is isomorphic to $\wdL(0)$ and $\wdM(\omega_2)$ has simple socle (see, e.g., \cite[Theorem 51]{CCM}) isomorphic to $\wdL(\la)$, we may conclude that $\Ext^1_{\mc O}(\wdL(0),\wdL(\la))\neq 0$. By Proposition \ref{prop::100}, it follows that $\wdL(0)$ is isomorphic to a submodule of $U_\alpha(\la)$. 
  	
  	Now, let 
  	\begin{align}
  	&0\rightarrow \wdL(\la) \rightarrow E\xrightarrow{f} \wdL(\omega_2)\rightarrow 0, \label{eq::exm11eq2}
  	\end{align} be a short exact sequence.  Since $\la$ can not be written as a sum of $\omega_2$ and positive roots, we may conclude that the preimage of the highest weight vector of $\wdL(\omega_2)$ under $f$ is again a highest weight vector of $E$, namely, $E$ is a quotient of $\wdM(\omega_2)$. By \cite[Lemma 6.1]{CC} and \cite[Lemma 5.11]{CC}, the socle of $\wdM(\omega_2)$ is $\wdL(\la)$, and we  have a non-split short exact sequence of the radical $\rad \wdM(\omega_2)$: 
  	\begin{align}
  	&0\rightarrow \wdL(\la) \rightarrow \rad\wdM(\omega_2)\rightarrow \wdL(0) \rightarrow 0, 
  	\end{align} which implies that \eqref{eq::exm11eq2} is  split. Namely, we have   $\Ext^1_{\mc O}(\wdL(\omega_2), \wdL(\la))= 0$. Consequently, the socle and top of  $U_\alpha(\la)$ are isomorphic to $\wdL(0)$ and  $\wdL(\omega_2)$, respectively.  
  	
  	{\em Method 2}. By \cite[Theorem 2.3]{AS} we have $T^\oa_{s_\alpha}(L(\la)) = M(s_\alpha\cdot\la)^\vee$, where $M(s_\alpha \cdot \la)^\vee$ denotes the dual Verma module in the sense of \cite[Section 3.3]{Hu08}. We claim that  \begin{align}&T_{s_\alpha} \wdL(\la)=T_{s_\alpha} K(L(\la))\cong K(M(s_\alpha\cdot\la)^\vee).\label{eq::37} %\cong K(T_{s_\alpha}^\oa M(\la)),   
  		\end{align} To see this, we note that $T^\oa_{s_\alpha}L(\la)= T^\oa_{s_\alpha}M(\la)\cong M(s_\alpha\cdot\la)^\vee$, which has a unique $\g_{\geq 0}$-module structure, and so $T_{s_\alpha} K(M(\la))=T_{s_\alpha} \Ind_{\g_{\geq 0}}^{\g}M(\la)\cong \Ind_{\g_{\geq 0}}^{\g}T_{s_\alpha}^\oa M(\la)\cong \Ind_{\g_{\geq 0}}^{\g}  M(s_\alpha\cdot \la)^\vee$, as desired.  
  	%We then compute   \begin{align}&T_{s_\alpha} \wdL(\la)=T_{s_\alpha} K(L(\la))\cong K(M(s_\alpha\cdot\la)^\vee) %\cong K(T_{s_\alpha}^\oa M(\la)),   	\end{align} which 
  	This implies that there is a short exact sequence 
  	$$0\rightarrow K(L(\omega_2)) \rightarrow T_{s_\alpha} \wdL(\la)\rightarrow \wdL(\la) \rightarrow 0.$$
  	Consequently, $U_\alpha(\la)$ is isomorphic to $K(L(\omega_2))$ with  socle $\wdL(0)$ and top $\wdL(\omega_2)$.
  \end{ex} 
  
  The Jantzen middle  $U_\alpha(\la)$ in Example \ref{ex::exm11} is non-semisimple. We will generalize this result and show the existence of non-semisimple Jantzen middles for arbitrary $\pn$ using an argument similar to Method 1 of Example \ref{ex::exm11}; see Proposition \ref{Thm::nonssU}. % Namely,  we will show  in  Proposition \ref{Thm::nonssU} that, for any block of the category $\mc O^\Z$ of $\pn$ which contains at least a simple module of atypical highest weight, there always exists a non-semisimple Jantzen middles   due to the lack of a simple preserving duality. 

  \subsubsection{Some technical tools} \label{sect::322} %{Kazhdan-Lusztig theory}
  
  For basic classical Lie superalgebras, it has been proved in \cite[Theorem 6.4]{Co16} that the Jantzen middles (associated with certain roots) of a block of $\mc O$ are semisimple if it admits a Kazhdan-Lusztig theory in the sense of \cite[Definition 3.3]{CPS1}. The main goal of this section is to prove an analogue for   Lie superalgebras in our setting, including $\pn$.  We then use this result to complete the proof of Theorem B  introduced in Section \ref{sec1}.

 The following lemma is an analogue of    \cite[Corollary 5.7]{Co16} for any   Lie superalgebra of type I-0.
  \begin{lem} \label{lem::14}
  	Let $\alpha \in \Pi_0$. Then for any   $\alpha$-free modules $M, V\in \mc O$ and $\widehat \alpha$-free module $N\in \mc O$, we have 
  	\begin{align}
  	&\Hom_{\mc O}(T_{s_\alpha} M, DN) \cong \Hom_{\mc O}(T_{s_{\widehat \alpha}}N, DM).\label{eq::lem141} \\
  	&\Ext^k_{\mc O}(T_{s_\alpha} M, T_{s_\alpha} V) \cong \Ext^k_{\mc O}(M, V),~\text{ for any }k\geq 0. \label{eq::lem142}	
  	\end{align} 
  \end{lem}
  \begin{proof}  We shall adapt the proof of \cite[Corollary 5.7]{Co16} to establish \eqref{eq::lem141}. We  compute   
  	\begin{align*}
  	&\Hom_{\mc O}(T_{s_\alpha} M, DN) \cong \Hom_{D^b(\mc O)}(\mc LT_{s_\alpha} M, DN)\\
  	&\cong \Hom_{D^b(\mc O)}(M, \mc RG_{s_\alpha} DN)\cong \Hom_{D^b(\mc O)}(M, D\mc LT_{s_{\widehat \alpha}}N) \\ &\cong \Hom_{\mc O}(M, DT_{s_{\widehat \alpha}}N) \cong  \Hom_{\mc O}(T_{s_{\widehat \alpha}}N, DM).
  	\end{align*}  
  	For any $k\in \Z$, let $(-)[k]$ be the corresponding shift functor on $\mc D^b(\mc O)$. To obtain \eqref{eq::lem142}, we use the argument as in proof of \cite[Corollary 5.7 (2)]{Co16} and  compute 
  	\begin{align*}
  	&\Ext^k_{\mc O}(T_{s_\alpha} M, T_{s_\alpha} V) \cong  \Hom_{D^b(\mc O)}(\mc LT_{s_\alpha} M, \mc LT_{s_\alpha} V[k]) \\
  	&\cong \Hom_{D^b(\mc O)}(M, V[k]) \cong \Ext^k_{\mc O}(M, V).
  	\end{align*}
  	The conclusion follows. 
  \end{proof}
  
  The following non-vanishing property of Ext-group will be helpful. 
  \begin{lem} \label{lem::lem14}
  	Suppose that $\Ext^1_{\mc O_\oa}(M(\la), L(\mu))\neq 0$,  for some $\la ,\mu \in \h^\ast$. Then we have $\Ext^1_{\mc O}(\wdM(\la), \wdL(\mu))\neq 0$.
  \end{lem}
  \begin{proof}
  	Let 
  	\begin{align}
  	&0\rightarrow L(\mu)\rightarrow E\rightarrow M(\la)\rightarrow 0\label{eq::lem143}
  	\end{align}be a non-split short exact sequence in $\mc O_\oa$. We note that every maximal submodule of $E$ contains $L(\mu)$ (otherwise \eqref{eq::lem143} is split). Since $M(\la)$ has a unique maximal submodule, we may conclude $E$ has a simple top. Applying the Kac functor $K(-)$, we obtain a short exact sequence in $\mc O$
  	\begin{align}
  	&0\rightarrow K(L(\mu))\rightarrow K(E)\rightarrow \wdM(\la)\rightarrow 0.\label{eq::lem144}
  	\end{align} By \cite[Theorem 51]{CCM}, the module $K(E)$ has  a simple top. Now, consider the short exact sequence
  	\begin{align}
  	&0\rightarrow \wdL(\mu)\rightarrow K(E)/\rad K(L(\mu))\rightarrow \wdM(\la)\rightarrow 0.\label{eq::lem145}
  	\end{align} Since $K(E)$ has a simple top, we may conclude that \eqref{eq::lem145} is non-split. The conclusion follows.  
  \end{proof}

  \subsubsection{Abstract Kazhdan-Lusztig theory} \label{sect::KLtheory} We recall the definition of an  abstract Kazhdan-Lusztig theory formulated by Cline, Parshall and Scott in \cite[Definition 3.3]{CPS93}. Let $\mc C$ be a highest weight category with weight poset $\Lambda$, simple objects $S(\la)$, induced objects $A(\la)$, Weyl objects $V(\la)$ and a length function $\ell:\Lambda \rightarrow \Z$ in the sense of \cite[Section 1]{CPS93} (see also \cite[Definition 2.1]{CPS932}). Then $\mc C$ is said to have  an {\em abstract Kazhdan-Lusztig theory} relative to $\ell$ provided that 
  \begin{align}
  &\Ext^n_{\mc C}(S(\la), A(\nu)) \neq 0 \Rightarrow \ell(\la) -\ell(\nu) \equiv n \text{ (mod $2$)},\\
  &\Ext^n_{\mc C}(V(\nu), S(\la)) \neq 0 \Rightarrow \ell(\la) -\ell(\nu) \equiv n \text{ (mod $2$)},
  \end{align} for any $n$ and $\la,\nu\in \Lambda$; see \cite{CPS93, CPS932, Scott1, Par} for  the background, examples and discussions. 
  
  For the category $\mc O^\Z$ of the general linear Lie superalgebra $\g=\gl(m|n)$,  it is proved in \cite[Corollary 3.3]{CSTAMS} that the Brundan-Kazhdan-Lusztig theory formulated in \cite{Br1} and established in \cite{CLW2, BLW} is an abstract  Kazhdan-Lusztig theory. For a weight $\eta \in \h^\ast$, we  denote by $\mc O_\eta$ the  indecomposable block of $\mc O$ that contains $\wdL(\eta)$. The following theorem is a restatement of  Theorem A, which is an analogue of \cite[Theorem 6.4]{Co16} for Lie superalgebras of type I-0. %Our proof of this theorem  is an adaption of \cite[Theorem 6.4]{Co16}.
   %We restate Theorem A as follows, which  makes a connection with Jantzen middles, and it  is  an analogue of \cite[Theorem 6.4]{Co16} for Lie superalgebras that satisfy \eqref{ass::A1}-\eqref{pre::tridec}. 

  \begin{thm} \label{thm::nonKL}
 %  	Let $\mc O_\eta$ be a indecomposable block of  $\mc O^\Z$. 
 Let $\eta$ be integral. Suppose that $\mc O_\eta$ contains a non-semisimple Jantzen middle. Then $\mc O_\eta$ does not admit a Kazhdan-Lusztig theory in the sense of \cite[Definition 3.3]{CPS93}.
  \end{thm}
  \begin{proof}
  	%We shall prove an analogue of \cite[Theorem 6.4]{Co16} for the periplectic Lie superalgebras $\pn$ by establishing all essential ingredients in the proof of \cite[Theorem 6.4]{Co16}. 
  	
  	%By Proposition \ref{Thm::nonssU}, it suffices to show that the Jantzen middles of a block  of $\mc O^\Z$ that admits a Kazhdan-Lusztig theory are always  semisimple. 
  	We shall adapt the proof of \cite[Theorem 6.4]{Co16}  to  obtain the conclusion for Lie superalgebra $\g$ of type I-0  by establishing all essential ingredients.  	To see this, let  $\alpha\in \Pi_0$ and $\la, \mu, \gamma \in \mc P$ with $\widehat \la \neq s_{\widehat \alpha}\cdot \widehat\mu$. We claim that if $\wdL(\la)$, $\wdL(\gamma)$ are $\alpha$-free and $\wdL(\mu)$ is $\alpha$-finite then we have  
  	\begin{align}
  	&\Hom_{\mc O}(U_\alpha(\la), \nbob(\gamma))  = \Hom_{\mc O}(\wdM(\gamma),U_\alpha(\la)) =0, \label{eq::nonKL1}
  	\end{align} and inclusions
  	\begin{align}
  	&\Hom_{\mc O}(U_\alpha(\la), \nbob(\mu))  \hookrightarrow  \Ext^1_{\mc O}(\wdL(\la), \nbob(\mu)),\label{eq::nonKL2} \\
  	&\Hom_{\mc O}(\wdM(\mu), U_\alpha(\la))  \hookrightarrow  \Ext^1_{\mc O}(\wdM(\mu), \wdL(\la)).\label{eq::nonKL3}
  	\end{align}
  	The equality \eqref{eq::nonKL1} is an immediate consequence of Proposition \ref{prop::prop10}. Now we are going to show \eqref{eq::nonKL2}. To see this, we apply the functor $\Hom_{\mc O}(-,\nbob(\mu))$ to the short exact sequence 
  	\begin{align}
  	&0\rightarrow U_\alpha(\la) \rightarrow  T_{s_\alpha}\wdL(\la) \rightarrow \wdL(\la)\rightarrow  0,
  	\end{align} and obtain a long exact sequence which contains 
  	\begin{align}
  	& \Hom_{\mc O}(T_{s_\alpha} \wdL(\la), \nbob(\mu)) \rightarrow \Hom_{\mc O}(U_\alpha(\la), \nbob(\mu)) \rightarrow \Ext^1_{\mc O}(\wdL(\la), \nbob(\mu)).
  	\end{align}

  	We recall the following isomorphisms  
  	$$\nbob(\mu)\cong D\wdM_{\mf b^r}(\widehat \mu),~ \widetilde L_{\mf b^r}(\widehat \la)\cong D\wdL(\la)$$ from \eqref{eq::dualVerma}. Also, we note that $\widehat \mu$ is $\widehat \alpha$-finite and so $T_{s_{\widehat \alpha}}^\oa M(\widehat \mu)\cong M(s_{\widehat \alpha}\cdot \widehat \mu)$, which implies that $T_{s_{\widehat \alpha}} \wdM_{\mf b^r}(\widehat \mu)\cong \wdM_{\mf b^r}(s_{\widehat{\alpha}}\cdot \widehat \mu)$ for the same reason as that given for \eqref{eq::37}. Therefore by Lemma \ref{lem::14} we have   
  	\begin{align}
  	&\Hom_{\mc O}(T_{s_\alpha} \wdL(\la), \nbob(\mu)) \cong \Hom_{\mc O}(\wdM_{\mf b^r}(s_{\widehat{\alpha}}\cdot \widehat \mu), \widetilde L_{\mf b^r}(\widehat \la))=0.
  	\end{align} This proves \eqref{eq::nonKL2}.

     Next, we proceed with the proof of \eqref{eq::nonKL3}.
     By \cite[Lemma 6.2]{AL} and \cite[Lemma 5.7]{CoM1}, we have a four term exact sequence 
     \begin{align}
     &0\rightarrow M(s_\alpha\cdot \mu) \rightarrow M(\mu) \rightarrow T_{s_\alpha}^\oa M(s_\alpha\cdot \mu) \rightarrow M(s_\alpha\cdot \mu)\rightarrow 0. 
     \end{align}
Applying the Kac functor $K(-)$, we obtain a four term exact sequence of $\g$-modules
\begin{align}
&0\rightarrow \wdM(s_\alpha\cdot \mu) \rightarrow \wdM(\mu) \rightarrow K(T_{s_\alpha}^\oa M(s_\alpha\cdot \mu)) \rightarrow \wdM(s_\alpha\cdot \mu)\rightarrow 0.  \label{eq::322}
\end{align} We claim that $K(T_{s_\alpha}^\oa M(s_\alpha\cdot \mu)) \cong T_{s_\alpha}K(M(s_\alpha\cdot \mu))$. To see this, let $M$ denote the $\g_{\geq 0}$-module which is $M(s_\alpha\cdot \mu)$ with trivial $\g_1$-action. With slightly abusing notations, we again denote by $T_{s_\alpha}$ the twisting functor for $\g_{\geq 0}$. Since $$\Res_{\g}^{\g_{\geq 0}}T_{s_\alpha}M\cong T_{s_\alpha}^\oa M(s_\alpha\cdot \mu)$$ is a quotient of $P(s_\alpha\cdot \mu)$, it follows that $\Res_{\g}^{\g_{\geq 0}}T_{s_\alpha}M$ is  indecomposable. Therefore, we have either $(T_{s_\alpha}M)_\oa=0$ or $(T_{s_\alpha}M)_\ob=0$, and so the $\g_1$-action on $T_{s_\alpha}M$ is trivial. Consequently, we have $K(T_{s_\alpha}^\oa M(s_\alpha\cdot \mu)) \cong \Ind_{\g_{\geq 0}}^\g T_{s_\alpha}M \cong T_{s_\alpha}K(M(s_\alpha\cdot \mu))$.  	Using \eqref{eq::322}, the equality  \eqref{eq::nonKL3}  follows from an argument identical to the proof of \cite[Corollary 6.5 (2)]{Co16}. 
  	
  	%Finally, we will adapt the argument as in the proof of \cite[Thoerem 6.4]{Co16}. To see this, we let $\mc O_\eta$ by an atypical block of $\mc O^\Z$ that 
  	
  	Suppose on the contrary that $\mc O_\eta$ 
  	admits a Kazhdan-Lusztig theory relative to a length function $\ell$ from the set of highest weights of simple modules of $\mc O_\eta$ to $\Z$. We claim that,  for given $\wdL(\mu_1),\wdL(\mu_2),U_\alpha(\la)\in \mc O_\eta$, if  
  	\begin{align}
  	&\Hom_{\mc O}(U_\alpha(\la), \nbob(\mu_1)) \neq 0, ~\Hom_{\mc O}(U_\alpha(\la), \nbob(\mu_2)) \neq 0,   \label{eq::eq423}
  	\end{align} then we have 
  	\begin{align}
  	&\Ext^1_{\mc O}(\wdM(\mu_1),\wdL(\mu_2))=0. \label{eq::eq424}
  	\end{align} To see this, we first assume that  $\widehat \la \neq s_{\widehat \alpha}\cdot \widehat\mu_1,~s_{\widehat \alpha}\cdot \widehat\mu_2$. Then $\Ext^1_{\mc O}(\wdL(\la), \nbob(\mu_1))\neq 0,~\Ext^1_{\mc O}(\wdL(\la), \nbob(\mu_2))\neq 0$ by \eqref{eq::nonKL2}, which implies that  $\ell(\mu_1)\equiv \ell(\mu_2)~(\text{mod}~2)$.  Hence \eqref{eq::eq424} follows.
  	
  	We now show that \eqref{eq::eq423} implies \eqref{eq::eq424}    in the case when $\widehat \la\in \{s_{\widehat \alpha}\cdot \widehat\mu_1,~ s_{\widehat \alpha}\cdot \widehat\mu_2\}$. To see this, we only need to consider the case that $\widehat \la = s_{\widehat \alpha}\cdot \widehat\mu_1$ and $\widehat \la \neq  s_{\widehat \alpha}\cdot \widehat\mu_2$. By \eqref{eq::nonKL2} again, we have $\Ext^1_{\mc O}(\wdL(\la), \nbob(\mu_2))\neq 0$, and so $\ell(\la)  \equiv \ell(\mu_2)+1 ~\text{(mod $2$)}$. Also, we note that   $\widehat \la =  s_{\widehat \alpha}\cdot \widehat\mu_1$ is equivalent to $\la =s_\alpha\cdot \mu_1$. It then follows from Lemma  \ref{lem::lem14} and \cite[Proposition 6.17]{CoM2} that $\Ext^1_{\mc O}(\wdM(\la), \wdL(\mu_1))> 0$, which implies that $\ell(\la)  \equiv \ell(\mu_1)+1 ~\text{(mod $2$)}$. Consequently, since  $\ell(\mu_1)\equiv \ell(\mu_2)~\text{(mod 2)}$ we have  $\Ext^1_{\mc O}(\wdM(\mu_1), \wdL(\mu_2))=0$, as desired.

  	Similarly, if  
  	\begin{align}
  	&\Hom_{\mc O}(\wdM(\mu_1), U_\alpha(\la))\neq 0,~ \Hom_{\mc O}(\wdM(\mu_2), U_\alpha(\la))\neq 0,
  	\end{align} then %$\Ext^1(\wdM(\mu_1),\widetilde L(\la))\neq 0,~\Ext^1(\wdM(\mu_2),\widetilde L(\la))\neq 0$ 
  	$\ell(\mu_1)\equiv \ell(\mu_2)~\text{(mod 2)}$ by \eqref{eq::nonKL3} for the same reason. We may conclude that $\Ext^1_{\mc O}(\widetilde L(\mu_1), \nbob(\mu_2))=0$. Consequently, the Jantzen middle $U_\alpha(\la)$ is always semisimple by \cite[Thoerem 4.1]{CPS93}, which   contradicts to our assumption. %Proposition \ref{Thm::nonssU}. 
  	This completes the proof. 
  \end{proof}

  \section{The Jantzen middles for $\pn$}\label{sect::KLpn}
    
  In this section, we consider the periplectic Lie superalgebra $\g:=\pn$. Recall that $\omega_n:= \vare_1 +\vare_2+\cdots +\vare_n$. % For   weight $\la \in \h^\ast$, we  denote by $\mc O_\la$ the indecomposable block of $\mc O$ that contains $L_\la$.
  
  \subsection{Jantzen middles in typical blocks} \label{Sect::421} %Without loss of generality, we may redefine the set $\mc P$ as  $\sum_{i=1}^n\Z\vare_i$, which we will still denote by $\mc P$. 
  %Motivated by the equivalence relation defined on $\h^\ast$ from \cite[Section 5.2]{CC}, 
  
  We  define the equivalence relation $\sim$ on $\mc P$  transitively generated by $\la\sim w\cdot \la$ and $\la\sim \la \pm2\vare_k$, for~$w\in W$ and $1\le k\le n$; see \cite[Section 5.2]{CC}.  %This $\sim$ is a restriction of the equivalence relation defined on $\h^\ast$ from \cite[Section 5.2]{CC} and used to describe the indecomposable blocks of the whole category $\mc O$, see \cite[Theorem 5.4]{CC}. 
    In particular, we have the following decomposition of $\mc O^\Z$ into indecomposable blocks. 
  \begin{lem} \label{lem::pnblocks} \emph{(}\cite[Theorem 5.4]{CC}\emph{)} 
  	 For any $\la, \mu \in \mc P$, we have 
  	 \begin{align}
  	 &\widetilde{L}(\mu)\in \mc O_\la \Leftrightarrow \la \sim \mu.\label{eq::CC54}
  	 \end{align} In particular, we have 
  	$$\cO^{\mZ}\;=\;\bigoplus_{i=0,~k\in \C}^n \cO_{\partial^i+k\omega_n},$$
  	with~$\partial^{i}:=i\vare_1+(i-1)\vare_2 + \cdots + \vare_{i}$.
  \end{lem} For any $k\in\C$, we have an equivalence  
\begin{align}
&\C_{k\omega_n}\otimes -:\cO_{\partial^i}\cong \cO_{\partial^i+k\omega_n}, \label{eq::equiv}
\end{align}
 see \cite[Lemma 5.10]{CC}. Therefore, we shall focus on the blocks $\mc O_{\partial^0},\ldots ,\mc O_{\partial^n}.$  We recall  the notion of typicality of weights for~$\mf{pe}(n)$ from \cite[Section 5]{Se02} below. Let $T, T_+,T_-$ be the polynomials on $\h^\ast$ given by  
  \begin{align}
  &T_\pm(\la):= \prod_{\alpha \in \Phi^+_\oa}(\langle \la+\rho, \alpha\rangle \pm 1) =\prod_{i<j}(\la_i-\la_j+j-i\pm 1),~T(\la):=T_+(\la) T_-(\la).
  \end{align} Then $\la$ is called {\em typical} if $T(\la)\neq 0$, and {\em atypical} otherwise. A block $\mc O_\la$ is said to be {\em atypical} if  $\widetilde{L}(\mu)\in\mc O_\la$, for some atypical $\mu\in \h^\ast$, and $\mc O_\la$ is called {\em typical} otherwise. By Lemma \ref{lem::pnblocks},  these blocks $$\mc O_{\partial^i},~\text{for } i=0,\ldots n-2$$ are atypical.  Conversely,  if $\wdL(\la)$ lies in one of  $$\mc O_{\partial^{n-1}},~\mc O_{\partial^{n}},$$ then $\la+\rho =\sum_{i=1}^n(\la+\rho)_i\vare_i$ satisfying that $(\la+\rho)_1,(\la+\rho)_2,\ldots,(\la+\rho)_n$ are either all even or all odd. Therefore, $\la$ is typical, and so both $\mc O_{\partial^{n-1}},~\mc O_{\partial^{n}}$ are typical. 
  
  It is shown in \cite[Theorem C, Theorem 4.6]{CP} that the characters of tilting modules in $\mc O_{\partial^{n-1}}$ and  $\mc O_{\partial^{n}}$ are completely controlled by the Kazhdan-Lusztig polynomials of type A Lie algebras. Therefore it is natural to ask whether $\mc O_{\partial^{n-1}}$ and  $\mc O_{\partial^{n}}$ admit abstract Kazhdan-Lusztig theories. The following proposition provides another evidence by showing the semisimplicity of Jantzen middles of $\mc O_{\partial^{n-1}}$ and  $\mc O_{\partial^{n}}$.
  
  \begin{thm} \label{prop::9}  Suppose that $\la\in \h^\ast$ is typical. Then  $U_\alpha(\la)$ is either zero or semisimple, for any $\alpha \in \Pi_0$. In particular, all (non-zero) Jantzen middles in $\mc O_{\partial^{n-1}}$ and  $\mc O_{\partial^{n}}$ are semisimple. %If $T_-(\la)\neq 0$ then $U_\alpha(\la)$ is semisimple. In particular, all Jantzen middles in $\mc O_{\partial^{n-1}}$ and  $\mc O_{\partial^{n}}$ are semisimple. 
  \end{thm}
  \begin{proof}
  	By  \cite[Lemma 3.2]{Se02} (see, also \cite[Lemma 5.11]{CC}), we have $\wdL(\la) = K(L(\la)).$ Set $s:=s_\alpha$. We first claim that $T_sK(L(\la))\cong K(T^\oa_sL(\la)).$ To see this, consider the  $\g_{\geq 0}$-module $L_\la$, which is the $\g_\oa$-module $L(\la)$ with trivial $\g_1$-action.  With slightly abusing notations, we denote by $T_s$ the twisting functor for $\g_{\geq 0}$-modules again. Note that $\Res_{\g_0}^{\g_{\geq 0}}T_sL_\la\cong T_s^\oa L(\la)$ is either zero or an indecomposable $\g_\oa$-module by Proposition \ref{prop::prop10}.
  	Since we have either $(T_sL_\la)_\oa =0$ or $(T_sL_\la)_\ob =0$, it follows  that $\g_1$ acts on the $\g_{\geq 0}$-module $T_sL_\la$ trivially. We compute 
  		\begin{align*}
  		&T_sK(L(\la))\cong \Ind_{\g_{\geq 0}}^{\g} T_s L_\la\cong K(T_s^\oa L(\la)).
  		\end{align*} 
  	
  	By  \cite[Theorem 51]{CCM}  the Kac functor $K(-)$ preserves the length of top. Consequently, we have
  	\[U_\alpha(\la) = \rad T_s\wdL(\la)\cong \rad K(T_s^\oa L(\la))\cong K(\rad T_s^\oa L(\la)).\]
  	Since $T_s^\oa(-)$ is right exact and $\ch T_s^\oa M(\la)=\ch M(s\cdot \la)$, we may conclude that every composition factor of $\rad T_s^\oa L(\la)$ is of the form $L(\mu)$ with $\mu\in W\cdot \la$. Therefore, if  $m_{\la,\mu}:=[\rad T_s^\oa L(\la):L(\mu)]>0$, then $\mu$ is typical. This implies that $U_\alpha(\la)\cong  \bigoplus_{\mu\in \h^\ast} \wdL(\mu)^{\oplus m_{\la,\mu}}$ is either zero or  semisimple. The conclusion follows. 
  \end{proof} 

\begin{ex}
	Consider $\g =\mf{pe}(2)$ with the unique positive even root $\alpha:=\vare_1-\vare_2$. We have computed Jantzen middles $U_\alpha(\la)$ for atypical weights  $\la$ in Example \ref{ex::exm11}. Now, we consider Jantzen middles $U_\alpha(\la)$ in the case when $\la$ are typical. Set $T:=T_{s_\alpha}$. 
	
	Suppose that  $\la=a\vare_1+b\vare_2\in \Z\vare_1\bigoplus \Z\vare_2$ is typical, namely, $b\neq a, a+2$. If $b>a+2$ then $T_s^\oa M(\la)$ is the dual Verma module over $\g_\oa$ with top and radical isomorphic to $L(\la)$ and $L(s\cdot \la)$, respectively. It then  follows that \[U_\alpha(\la) = \rad T_sL(\la)\cong  K(\rad T_s^\oa M(\la))\cong \wdL(s\cdot \la).\]	
   
   If $b=a+1$	then  $U_\alpha(\la)=0$  by the character formulas of twisted simple modules computed in \cite[Lemma 6.7]{CC}. We now collect our results in the following complete classification of Jantzen middles of $\mf{pe}(2)$ (see also \cite[Lemma 6.7]{CC}):

   	\begin{align}
   &U_\alpha(\la)=\begin{cases}  0  &\mbox{ if $b\leq a+1$,}\\
   K(L(\la+\vare_1-\vare_2))&\mbox{ if $b=a+2$.} \\
  \wdL(s\cdot \la)&\mbox{ if $b>a+2$.} 
   \end{cases} 
   \end{align}
\end{ex}

  %It is worth pointing out that every block of $\mc O^\Z$ for $\pn$ can not be  equivalent to a block of the BGG category of a Lie superalgebra equipped with a  simple-preserving duality.  To see this, we note that  every block of $\mc O^\Z$ contains a projective cover which is injective by Lemma \ref{lem::pnblocks}. By \cite[Lemma 4.4]{CCC}, such an indecomposable projective-injective module has non-isomorphic top and socle. The conclusion follows.
  
  \subsection{The absence of   Kazhdan-Lusztig theory for $\pn$} In this section, we turn to the semisimplicity   of Jantzen middles for atypical blocks of  $\pn$. 

 Suppose that $\mc O_\eta$ is an atypical block of $\mc O^\Z$. Applying the equivalence \eqref{eq::equiv}  if necessary,  we may assume that $\eta \in \bigoplus_{i=1}^n\Z\vare_i$. 	For any weight $\gamma\in \mc P$, we let $\gamma_i$ be determined by $\gamma= \sum_{i=1}^n \gamma_i\vare_i$.  
 	Since  $\mc O_\eta$ is atypical, there are $1\leq i,j\leq n$ such that $(\eta+\rho)_i  \not\equiv (\eta+\rho)_j~ \text{(mod~2)}$ by \eqref{eq::CC54}. In addition,   there is an anti-dominant weight $\la\in \mc P$ such that  $\wdL(\la)\in \mc O_\eta$ satisfying $$(\la+\rho)_1 =0,~  (\la+\rho)_2 =1,\text{ and } \la_{k+1} -\la_{k}>1,$$ for $k=2,3,\ldots,n-1.$ The following proposition shows that the corresponding Jantzen middle $U_\alpha(\la)$ in $\mc O_\eta$ is non-semisimple.

  \begin{prop} \label{Thm::nonssU}
  	%Suppose that $\mc O_\eta$ is an atypical block of $\mc O^\Z$. Then  there is a non-semisimple Jantzen middle $U_\alpha(\la)$ in $\mc O_\eta$. 
  	 Retain the notations above. Then $U_\alpha(\la)$ is not semisimple.

  	% $\emph{Ext}^1(\widetilde{L}(\la),\widetilde{L}(\mu)) = 0$ and  $\emph{Ext}^1(\widetilde{L}(\mu),\widetilde{L}(\la)) \neq 0$, for some $\widetilde{L}(\mu)\in \mc O_\eta$.
  \end{prop}
  \begin{proof} % Without loss of generality, we may assume that $\eta \in \bigoplus_{i=1}^n\Z\vare_i$.	Let $\alpha := \vare_1 -\vare_2$ and $s:=s_{\alpha}$. For any weight $\gamma\in \mc P$, we let $\gamma_i$ be determined by $\gamma= \sum_{i=1}^n \gamma_i\vare_i$.    	Since  $\mc O_\eta$ is atypical, there are $i,j$ such that $(\eta+\rho)_i  \not\equiv (\eta+\rho)_j~ \text{(mod~2)}$ by \eqref{eq::CC54}. In addition,   there is an anti-dominant weight $\la\in \mc P$ such that  $\wdL(\la)\in \mc O_\eta$, $(\la+\rho)_1 =0,~  (\la+\rho)_2 =1 $ and $\la_{k+1} -\la_{k}>1$, for $k=2,3,\ldots,n-1.$ 
  	
  	Let $\alpha := \vare_1 -\vare_2$ and $s:=s_{\alpha}$.
  	First, we note that the socle  
  	of $K(L(s\cdot \la))$ is $\wdL_{\mf b^r}(s\cdot \la- (n-1)\omega_n)$.   Put $\mu:= s \cdot \la -\vare_1-\vare_2=\la -2\vare_2$. It follows from Lemma \ref{lem::oddref} (see also Example \ref{ex::example3}) that  
  	$${\rm soc} K(L(s\cdot \la)) =\wdL_{\mf b^r}(s\cdot \la- (n-1)\omega_n) =   \wdL(\mu).$$
  	
   We firstly claim that $\Ext^1_{\mc O}(\wdL(\la), \wdL(\mu)) =0$, which will imply that  $[\text{top} (U_\alpha(\la)): \wdL(\mu) ] =0$ by Proposition \ref{prop::100}. To see this, we let 
  	\begin{align}
  	& 0\rightarrow \wdL(\mu)\rightarrow E \xrightarrow{\pi} \wdL(\la) \rightarrow 0, \label{eq::nonssUeq1}
  	\end{align}
  	be a short exact sequence. Let $v_\la \in E$ be the preimage of the highest weight vector of $\wdL(\la)$ under $\pi$. By weight consideration we have $\mf n^+ v_\la=0$. This means that $E$ is an image of the Verma module $\widetilde M(\la)$. Since $\la$ is anti-dominant, it follows that  $\widetilde M(\la) = \wdL(\la)$ by \cite[Lemma 5.11]{CC}; see also \cite[Lemma 3.2]{Se02}. Consequently, the short exact sequence \eqref{eq::nonssUeq1} is split.  
  	
  	Next, we show that   $\Ext^1_{\mc O}(\wdL(\mu), \wdL(\la)) \neq 0$, which will imply that  $[\text{soc}U_\alpha(\la): \wdL(\mu) ]\neq 0$ by Proposition \ref{prop::100}.  We recall that every Verma module has a simple socle by \cite[Theorem 51]{CCM}. Applying the Kac functor $K(-)$ to the short exact sequence $0\rightarrow L(\la)\rightarrow M(s\cdot \la)\rightarrow L(s\cdot \la) \rightarrow 0$, we then obtain a non-split short exact sequence
  	\begin{align*}
  	&0\rightarrow \wdL(\la) \rightarrow \wdM(s\cdot \la) \rightarrow K(L(s\cdot \la)) \rightarrow 0,  \end{align*} 
  	and so $ \wdL(\la)=\text{soc}\wdM(s\cdot \la)$.

  	Since $\ch T_{s} \wdL(\la) = \ch \wdM(s\cdot \la)$,  it follows that $\text{top} (U_\alpha(\la))\neq U_\alpha(\la)$. Hence, we have already proved that $U_\alpha(\la)$ is not semisimple. 
  	
  	We continue to show the stronger claim that $\Ext^1_{\mc O}(\wdL(\mu), \wdL(\la)) \neq 0$. Since $\wdM(s\cdot \la)$ has simple socle, there is a  submodule $E \subseteq \wdM(s\cdot \la)$ and a non-split short exact sequence
  	\begin{align}
  	&0\rightarrow \wdL(\la) \rightarrow E \rightarrow \wdL(\mu) \rightarrow 0. \label{Eq::ExtExt1}\end{align} This completes the proof. 
  \end{proof}
  
  %\begin{rem}	It follows from Proposition \ref{Thm::nonssU} that every atypical block of $\pn$ can not be  equivalent to a block of the BGG category of a Lie superalgebra  equipped with a  simple-preserving duality. It is worthwhile to pointing out that  this is also true for blocks $\mc O_{\partial^{n-1}}$ and $\mc O_{\partial^{n}}$. To see this, we note every block of $\mc O^\Z$ contains a projective cover which is injective by Lemma \ref{lem::pnblocks}. By \cite[Lemma 4.4]{CCC}, such an indecomposable projective-injective module has non-isomorphic top and socle. The conclusion follows.     \end{rem}

  The conclusion of Theorem B is a direct consequence of  Theorem \ref{thm::nonKL}, Theorem \ref{prop::9} and Proposition \ref{Thm::nonssU}.
  %\begin{cor} \label{cor::nonKLpn}	Consider $\g=\pn$. Then each atypical block of $\mc O^\Z$ does not admit a Kazhdan-Lusztig theory in the sense of \cite[Definition 3.3]{CPS93}.   \end{cor}
  
  %{\bf Question}: Is it true that typical blocks admit abstract  Kazhdan-Lusztig theories?\footnote{Add an appendix to introduce the odd reflections algorithms...}

\vspace{2mm}

\noindent 
Chih-Whi Chen:~Department of Mathematics, National Central University, Zhongli District, Taoyuan City, Taiwan;
E-mail: {\tt cwchen@math.ncu.edu.tw}
\hspace{2cm}

%%%%%%%%%%%%%%%%%%%%%%%%%%%%%%%%%%%%%%%%%%%%%%%%%%%%%%%%%%%%%%%%%%%%%%%%%%%%%%%

%%%%%%%%%

\end{document}